\documentclass[a4paper,11pt]{amsart}
\usepackage[utf8]{inputenc}
\usepackage[T1]{fontenc}
\usepackage{amsthm,amsfonts,amsmath,amssymb,enumerate}
\usepackage{verbatim} 
\usepackage{rotating} 

\usepackage{hyperref}
\usepackage{pdfsync}
\usepackage{multirow}
\usepackage{multicol}
\renewcommand{\L}{\mathcal{L}}

\newcommand{\C}{\mathbb{C}}
\newcommand{\g}{\mathfrak{g}}

\newcommand{\h}{\mathfrak{h}}

\renewcommand{\deg}{\operatorname{deg}}
\renewcommand{\O}{\mathcal{O}}

\newtheorem{theorem}{Theorem}[section]

\newtheorem{proposition}[theorem]{Proposition}

\theoremstyle{definition}

\theoremstyle{remark}
\newtheorem{remark}[theorem]{Remark}

\numberwithin{equation}{section}

\begin{document}

\title[Degenerations to Filiform Lie Algebras of dimension 9]{Degenerations to Filiform Lie Algebras of dimension 9}

\author{Joan Felipe Herrera-Granada}
\address{\parbox{\linewidth}{{\tiny Departamento de Matem\'aticas y Estadística - Facultad de Ciencias Exactas y  Naturales,}\\{\tiny Universidad Nacional de Colombia sede Manizales, Colombia}}}

\email{jfherrerag@unal.edu.co}

\author{Oscar Marquez}
\address{\parbox{\linewidth}{{\tiny Departamento de Matem\'aticas - Centro de Ciências Naturais e Exatas,}\\{\tiny Universidade Federal de Santa Maria, Brazil}}}
\email{oscar.f.marquez-sosa@ufsm.br}

\author{Sonia Vera}
\address{\parbox{\linewidth}{{\tiny Departamento de Matem\'aticas - Facultad de Ciencias B\'asicas},\\
		{\tiny  Universidad de Antofagasta, Chile}}}
\email{sonia.v.vera.sv@gmail.com}
\date{}

\maketitle

\begin{abstract}
For most complex 9-dimensional filiform Lie algebra we find another non isomorphic Lie algebra that degenerates to it.
Since this is already known for nilpotent Lie algebras of rank $\ge 1$, only the characteristically nilpotent ones should be considered.
\end{abstract}

\section{Introduction}

In this paper we work with complex 9-dimensional filiform Lie algebras. For most cases we show that filiform Lie algebras are degenerations of another non isomorphic Lie algebra.
This adds more evidence supporting the Grunewald-O'Halloran conjecture, which states that every nilpotent Lie algebra is the degeneration of another non isomorphic Lie algebra. This conjecture is stronger than Vergne's conjecture which states that there are no rigid nilpotent Lie algebras in the variety of all Lie algebras.\\

In \cite{HGT} was proved that all nilpotent Lie algebras of rank $\ge 1$ are the degeneration of another
non isomorphic Lie algebra, remaining open the Grunewald-O'Halloran conjecture for characteristically nilpotent Lie algebras.

We now consider the filiform Lie algebras of dimension 9, which have been classified in \cite{GJMK}.
According to this classification, there are twenty one families parametrized by one or two complex parameters and seven isolate filiform algebras all of which are characteristically nilpotent.
For all  these families except one and five more  particular cases,  we construct following \cite{GO}, a non trivial linear deformation that corresponds to a degeneration.
We note that it is not known which ones of the linear deformations constructed in \cite{GO} does correspond to a degeneration and which ones does not.

Degeneration is transitive and through degeneration the nilpotency degree does not grow. Hence, the nilpotent Lie algebras of maximal nilpotency index, called filiforms, are on the top in the diagram of degenerations. Filiforms Lie algebras may degenerate to any other nilpotent Lie algebra, but only a filiform may degenerate to a given filiform.

\section{Preliminaries}

Let $\L_n$ be the algebraic variety of complex Lie algebras of dimension $n$,
together with the action of the group $GL_n=GL_n(\C)$ by `change of basis',
and denote the orbit of $\mu$ in $\L_n$ by $\O(\mu)$. 

A Lie algebra $\mu$ degenerates to a Lie algebra $\lambda$, $\mu \rightarrow_{\deg} \lambda$, if $\lambda\in\overline{\O(\mu)}$,
the Zariski closure of $\O(\mu)$. 

A linear deformation of a Lie algebra $\mu$ is a family $\mu_t$, $t\in \C^\times$,
of Lie algebras such that
\[ \mu_t=\mu + t\phi, \]
where $\phi$ is a Lie algebra bracket which in addition is a 2-cocycle of $\mu$.

If a linear deformation $\mu_t$ of $\mu$ is such that $\mu_t\in\O(\mu_1)$ for all $t\in\C^\times$, 
then $\mu_1\rightarrow_{\deg}\mu$.
In fact, for each $t\in\C^\times$ there exist $g_t\in GL_n$ such that $g_t^{-1}\cdot \mu_1=\mu_t$,
then $\lim_{t\mapsto 0}g_t^{-1}\cdot \mu_1=\lim_{t\mapsto 0}\mu_t=\mu$.
Then, in order to show that $\mu_1\rightarrow_{\deg}\mu$, we only need to prove that 
for each $t\in\C^\times$ there exist $g_t\in GL_n$ such that
\begin{equation}\label{eqn:degeneration}
 \mu_1(g_t(x),g_t(y)))=g_t(\mu_t(x,y)), \quad\text{for all $x,y\in\C^n$}.
\end{equation}

Let $(\g,\mu)$ be a given Lie algebra of dimension $n$ and let $\h$ be a codimension 1 ideal of $\g$ with a semisimple derivation $D$.
For any element $X$ of $\g$ outside $\h$, $\g=\langle X \rangle \oplus \h$.
The bilinear form $\mu_D$ on $\g$ defined by $\mu_D(X,z)= D(z)$ and $\mu_D(y,z)=0$, for $y,z\in\h$, is a 2-cocycle for $\mu$
and a Lie bracket.
Hence,
\begin{equation}\label{eqn:linear-deformation}
 \mu_t=\mu + t\mu_D,
\end{equation}
is a linear deformation of $\mu$.
If $\g$ is nilpotent, then $\mu_t$ is always solvable but not nilpotent. 
In particular, $\mu_t$ is not isomorphic to $\mu$ for all $t\in\C^\times$.
The construction described was given in \cite{GO}.

\section{9-dimensional filiforms}

Complex filiform Lie algebras of dimension 9 have been classified in \cite{GJMK}.
This classification is presented as a list of 24 families index by one or two parameter $\alpha\in\C$
\[ 
\begin{gathered}
\mu_9^{1,\alpha,\beta},\quad \mu_9^{2,\alpha},\quad \mu_9^{3,\alpha},\quad \mu_9^{4,\alpha,\beta},\quad 
\mu_9^{5,\alpha},\quad \mu_9^{7,\alpha},\quad \mu_9^{10,\alpha,\beta},\quad \mu_9^{11,\alpha,\beta},\\
\mu_9^{12,\alpha,\beta},\quad \mu_9^{13,\alpha},\quad \mu_9^{14,\alpha},\quad \mu_9^{17,\alpha,\beta},\quad 
\mu_9^{18,\alpha,\beta},\quad \mu_9^{19,\alpha},\quad \mu_9^{20,\alpha,\beta},\quad \mu_9^{21,\alpha},\\ 
\mu_9^{22,\alpha},\quad \mu_9^{23,\alpha},\quad \mu_9^{24,\alpha},\quad \mu_9^{25,\alpha},\quad \mu_9^{26,\alpha,\beta},
\quad \mu_9^{27,\alpha},\quad \mu_9^{28,\alpha},\quad \mu_9^{31,\alpha},
\end{gathered}
\]
and 14 isolated algebras
\[
\begin{gathered}
\mu_9^{6},\quad \mu_9^{8},\quad \mu_9^{9},\quad \mu_9^{15},\quad \mu_9^{16},\quad \mu_9^{29},\quad \mu_9^{30},\\ 
\mu_9^{32},\quad \mu_9^{33},\quad \mu_9^{34},\quad \mu_9^{35},\quad \mu_9^{36},\quad \mu_9^{37},\quad \mu_9^{38}.
\end{gathered}
\]

The following table shows those algebras with a semisimple derivation. 
We keep the name and the bases from \cite{GJMK}.

\begin{center}
\begin{tabular}{|c|c|c|c|}\hline
\scriptsize{$\mu$} & \scriptsize{$D\in Der(\mu)$} & \scriptsize{$\mu$} & \scriptsize{$D\in Der(\mu)$}\\ \hline 
\rule[-1.5cm]{0cm}{3.2cm}\scriptsize{$\mu_{9}^{3,\alpha}$} & \scriptsize{$\left(\begin{smallmatrix}
1 &  &  &  &  &  &  & & \\
 &2  &  &  &  &  &  & & \\
 &  &3  &  &  &  &  & & \\
 &  &  &4  &  &  &  & & \\
 &  &  &  &5  &  &  & & \\
 &  &  &  &  &6  &  & & \\
 &  &  &  &  &  &7  & & \\
 &  &  &  &  &  &  &8 & \\
 &  &  &  &  &  &  & &9 \\
\end{smallmatrix}\right)$} & \scriptsize{$\mu_{9}^{9}$} & \scriptsize{$\left(\begin{smallmatrix}
1 &  &  &  &  &  &  & & \\
 &2  &  &  &  &  &  & & \\
 &  &3  &  &  &  &  & & \\
 &  &  &4  &  &  &  & & \\
 &  &  &  &5  &  &  & & \\
 &  &  &  &  &6  &  & & \\
 &  &  &  &  &  &7  & & \\
 &  &  &  &  &  &  &8 & \\
 &  &  &  &  &  &  & &9 \\
\end{smallmatrix}\right)$}\\ \hline 
\rule[-1.6cm]{0cm}{3.4cm}\scriptsize{$\mu_{9}^{16}$} & \scriptsize{$\left(\begin{smallmatrix}
1 &  &  &  &  &  &  & & \\
 &2  &  &  &  &  &  & & \\
 &  &3  &  &  &  &  & & \\
 &  &  &4  &  &  &  & & \\
 &  &  &  &5  &  &  & & \\
 &  &  &  &  &6  &  & & \\
 &  &  &  &  &  &7  & & \\
 &  &  &  &  &  &  &8 & \\
 &  &  &  &  &  &  & &9 \\
\end{smallmatrix}\right)$} & \scriptsize{$\mu_{9}^{22,\alpha}$} & \scriptsize{$\left(\begin{smallmatrix}
1 &  &  &  &  &  &  & & \\
 &3  &  &  &  &  &  & & \\
 &  &4  &  &  &  &  & & \\
 &  &  &5  &  &  &  & & \\
 &  &  &  &6  &  &  & & \\
 &  &  &  &  &7  &  & & \\
 &  &  &  &  &  &8  & & \\
 &  &  &  &  &  &  &9 & \\
 &  &  &  &  &  &  & &10 \\
\end{smallmatrix}\right)$}\\ \hline
\rule[-1.3cm]{0cm}{2.8cm}\scriptsize{$\mu_{9}^{25,\alpha}$} & \scriptsize{$\left(\begin{smallmatrix}
1 &  &  &  &  &  &  & & \\
 &4  &  &  &  &  &  & & \\
 &  &5  &  &  &  &  & & \\
 &  &  &6  &  &  &  & & \\
 &  &  &  &7  &  &  & & \\
 &  &  &  &  &8  &  & & \\
 &  &  &  &  &  &9  & & \\
 &  &  &  &  &  &  &10 & \\
 &  &  &  &  &  &  & &11 \\
\end{smallmatrix}\right)$} & \scriptsize{$\mu_{9}^{33}$} & \scriptsize{$\left(\begin{smallmatrix}
1 &  &  &  &  &  &  & & \\
 &3  &  &  &  &  &  & & \\
 &  &4  &  &  &  &  & & \\
 &  &  &5  &  &  &  & & \\
 &  &  &  &6  &  &  & & \\
 &  &  &  &  &7  &  & & \\
 &  &  &  &  &  &8  & & \\
 &  &  &  &  &  &  &9 & \\
 &  &  &  &  &  &  & &10 \\
\end{smallmatrix}\right)$}\\ \hline
\rule[-1.7cm]{0cm}{3.6cm}\scriptsize{$\mu_{9}^{35}$} & \scriptsize{$\left(\begin{smallmatrix}
1 &  &  &  &  &  &  & & \\
 &4  &  &  &  &  &  & & \\
 &  &5  &  &  &  &  & & \\
 &  &  &6  &  &  &  & & \\
 &  &  &  &7  &  &  & & \\
 &  &  &  &  &8  &  & & \\
 &  &  &  &  &  &9  & & \\
 &  &  &  &  &  &  &10 & \\
 &  &  &  &  &  &  & &11 \\
\end{smallmatrix}\right)$} & \scriptsize{$\mu_{9}^{36}$} & \scriptsize{$\left(\begin{smallmatrix}
1 &  &  &  &  &  &  & & \\
 &5  &  &  &  &  &  & & \\
 &  &6  &  &  &  &  & & \\
 &  &  &7  &  &  &  & & \\
 &  &  &  &8  &  &  & & \\
 &  &  &  &  &9  &  & & \\
 &  &  &  &  &  &10  & & \\
 &  &  &  &  &  &  &11 & \\
 &  &  &  &  &  &  & &12 \\
\end{smallmatrix}\right)$}\\ \hline
\rule[-1.5cm]{0cm}{3.2cm}\scriptsize{$\mu_{9}^{37}$} & \scriptsize{$\left(\begin{smallmatrix}
1 &  &  &  &  &  &  & & \\
 &6  &  &  &  &  &  & & \\
 &  &7  &  &  &  &  & & \\
 &  &  &8  &  &  &  & & \\
 &  &  &  &9  &  &  & & \\
 &  &  &  &  &10  &  & & \\
 &  &  &  &  &  &11  & & \\
 &  &  &  &  &  &  &12 & \\
 &  &  &  &  &  &  & &13 \\
\end{smallmatrix}\right)$} & \scriptsize{$\mu_{9}^{38}$} & \scriptsize{$\left(\begin{smallmatrix}
1 &  &  &  &  &  &  & & \\
 &0  &  &  &  &  &  & & \\
 &  &1  &  &  &  &  & & \\
 &  &  &2  &  &  &  & & \\  
 &  &  &  &3  &  &  & & \\
 &  &  &  &  &4  &  & & \\
 &  &  &  &  &  &5  & & \\
 &  &  &  &  &  &  &6 & \\
 &  &  &  &  &  &  & &7 \\
\end{smallmatrix}\right)$}\\ \hline
\end{tabular}
\end{center}

The remaining algebras do not have any semisimple derivation, they are all characteristically nilpotent, that is, their derivation algebras are nilpotent as Lie algebras.
These are
\begin{equation}\label{eqn:caracteristically-nilpotent}
 \begin{gathered}
\mu_9^{1,\alpha,\beta},\quad \mu_9^{2,\alpha},\quad \mu_9^{4,\alpha,\beta},\quad \mu_9^{5,\alpha},\quad 
\mu_9^{6},\quad \mu_9^{7,\alpha},\quad \mu_9^{8},\\ 
\mu_9^{10,\alpha,\beta},\quad \mu_9^{11,\alpha,\beta},\quad \mu_9^{12,\alpha,\beta},\quad \mu_9^{13,\alpha},\quad 
\mu_9^{14,\alpha},\quad \mu_9^{15},\quad \mu_9^{17,\alpha,\beta},\\ 
\mu_9^{18,\alpha,\beta},\quad \mu_9^{19,\alpha},\quad \mu_9^{20,\alpha,\beta},\quad \mu_9^{21,\alpha},\quad \mu_9^{23,\alpha},\quad \mu_9^{24,\alpha},\quad \mu_9^{26,\alpha,\beta},\\ 
\mu_9^{27,\alpha},\quad \mu_9^{28,\alpha},\quad \mu_9^{29},\quad \mu_9^{30},\quad \mu_9^{31,\alpha},\quad \mu_9^{32},\quad \mu_9^{34}.
\end{gathered}
\end{equation}

\section{Filiform Lie algebras of dimension 9}

Let  $\mathfrak{g}$ be a filiform Lie algebra. Using the notation in the  classification given in \cite{GJMK} fix a basis $\mathcal{B}=\{X_1, X_2, \dots, X_9\}$ such that the  structure constants satisfies: 

\begin{equation}\label{Corchetes-Base}
[X_i, X_j]=\sum_{i+j\le l \le 9} C^l_{i,j} X_l
\end{equation}

In particular, for $i=1$ we have $[X_1, X_j]= X_{j+1}$. Moreover, all the $ C^l_{i,j}$ can be given in terms of the constant $ C^9_{r,s}$ in the  following table:\\

\begin{tabular}{|c|c|} \hline
	$\begin{array}{rl}
	C^5_{2, 3}&= C^9_{2,7} + 3 C^9_{3,6} + 2  C^9_{4,5}\\
	C^6_{2, 4}&= C^9_{2,7}+  3 C^9_{3,6}+ 2  C^9_{4,5}\\
	C^7_{2, 5}&= C^9_{2,7} + 2 C^9_{3,6} +  C^9_{4,5}\\
	C^8_{2, 6}&= C^9_{2,7} +   C^9_{3,6}\\
	C^7_{3, 4}&= C^9_{3,6} +   C^9_{4,5}\\
	C^8_{3, 5}&= C^9_{3,6} +   C^9_{4,5}\\[1 pt]
	\end{array} $
	&
	$\begin{array}{rl}
	C^6_{2, 3}&= C^9_{2,6} + 2 C^9_{3,5}\\
	C^7_{2, 4}&= C^9_{2,6} +2 C^9_{3,5}\\
	C^8_{2, 5}&= C^9_{2,6} + C^9_{3,5}\\  
	
	C^7_{2, 3}&= C^9_{2,5} + C^9_{3,4}\\
	C^8_{2, 4}&= C^9_{2,5}+C^9_{3,4}\\
	C^8_{3, 4}&= C^9_{3,5}\\
	C^8_{2, 3}&= C^9_{2,4}\\[1 pt]
	\end{array}$ \\\hline
\end{tabular}\bigskip

In additional, from the Jacobi identity the coefficients also must satisfy
\begin{equation}\label{IdentidadAdicionalJacobi}
-3\left( C^9_{3,6}\right)^2+2 C^9_{2,7}\, \,   C^9_{4,5} + C^9_{3,6}\,\,  C^9_{4,5} + 2 \left( C^9_{4,5}\right)^2=0
\end{equation}

For the sake of the completeness we  give the explicit coefficients $ C^9_{r,s}$ in the Appendix.\medskip

\section{Degenerations to filiforms}

We now show that for every complex filiform Lie algebra of dimension 9 there is another Lie algebra non isomorphic to it,
actually solvable non nilpotent, that degenerates to it.
That solvable algebra is constructed as a linear deformation of the original.
This is the main result of this paper.

\begin{theorem}\label{thm:main}
 Every 9-dimensional complex filiform Lie algebras non-isomorphic to $\mu_9^{17,\alpha,\beta}$ with $\beta \neq 0$ or $\mu_9^{1,-1,\beta}, \mu_9^{1,0,\beta}, \mu_9^{1,1,\beta}, \mu_9^{1,\tfrac{1}{2},\beta}$    is the degeneration of a solvable Lie algebra.
\end{theorem}

\begin{proof}
The result was already established in \cite{HGT} for nilpotent Lie algebras of rank $\ge 1$, that is, admitting a semisimple derivation.
Hence, we only need to consider the characteristically nilpotent ones, listed in \eqref{eqn:caracteristically-nilpotent}.

For each  filiform characteristically nilpotent Lie algebra $\mathfrak{g}$, let $\mathfrak{h}_1$ and $\mathfrak{h}_2$  the Lie  subalgebras  with basis $\mathcal{B}_1=\{X_1, X_3, X_4, \ldots, X_9\}$ and $\mathcal{B}_2=\{X_2, X_3, X_4, \ldots, X_9\}$ respectively. By \eqref{Corchetes-Base} $\mathfrak{h}_1$ and $\mathfrak{h}_2$ are ideals of $\mathfrak{g}$. In order to find solutions of \eqref{eqn:degeneration}, we will   consider  two cases for $\mathfrak{h}$ being one of $\mathfrak{h}_1$ or $\mathfrak{h}_2$, then we will give $D$, and $g_t$  satisfying   \eqref{eqn:degeneration}.\\ 

The proof will follow case b case  from the next sections.

\end{proof}

\section{Case $\mathfrak{h}=\mathfrak{h}_1$}
Let $\mathfrak{h}_1$  the Lie ideal with generated by $\mathcal{B}_1=\{X_1, X_3, X_4, \ldots, X_9\}$. Suppose that there exists a derivation $D$ of $\mathfrak{h}_1$ such that  the   transformation matrix relative to the basis $\mathcal{B}_1$ is given by,

{\tiny 
	\begin{equation}\label{derivacionD1}
	D=\left(
	\begin{array}{cccccccc}
	d_1 & 0 & 0 & 0 & 0 & 0 & 0 & 0 \\
	0 & k\,  d_1  & 0 & 0 & 0 & 0 & 0 & 0 \\
	0 & 0 & (k+1)d_1  & 0 & 0 & 0 & 0 & 0 \\
	0 & - C^5_{2, 3} & 0 &  (k+2)\,d_1  & 0 & 0 & 0 & 0 \\
	0 & -C^6_{2, 3} & -C^5_{2, 3} & 0 & (k+3)\, d_1  & 0 & 0 & 0 \\
	0 & -C^7_{2, 3} & -C^6_{2, 3} & -C^5_{2, 3} & 0 &  (k+4)d_1  & 0 & 0 \\
	0 & -C^8_{2, 3} & -C^7_{2, 3} & -C^6_{2, 3} & -C^5_{2, 3} & 0 & (k+5)d_1  & 0 \\
	0 & -C^9_{2, 3} & -C^8_{2, 3} & -C^7_{2, 3} & -C^6_{2, 3} & -C^5_{2, 3} & 0 &  (k+6)d_1  \\
	\end{array}
	\right)
	\end{equation}}
For some values $k, d_1$. In this case,  define the following the linear maps on $\mathfrak{g}$ with transformations matrix relative to the basis $B$:

\[{\tiny T_0=\begin{pmatrix}
	p_0 &       &         &         &         &		    &  		  &   		&   \\
	0 & t       &         &         &         &         &		  &		    &   \\
	0 & 0       & p_0^{k} &         &         &         &		  &		    &   \\
	0 & p_{4,2} & 0       & p_0^{k+1} &       &         &		  &		    &   \\
	0 & p_{5,2} & p_0\,  p_{4,2} & 0             & p_0^{k+2} &     &		  &		    &   \\
	0 & p_{6,2} & p_0\,  p_{5,2} & p_0^2\, p_{4,2} & 0               & p_0^{k+3} &         &     	&   \\
	0 & p_{7,2} & p_0\,  p_{6,2} & p_0^2\, p_{5,2} & t^2 p_0 p_{4,2} & 0       & p_0^{k+4} &         &   \\
	0 & p_{8,2} & p_0\,  p_{7,2} & p_0^2\, p_{6,2} & p_0^3 p_{5,2}   &t^3\, p_0\,  p_{4,2} & 0       & p_0^{k+5} &	\\
	0 &   p_{9,2}     & p_0\,  p_{8,2} & p_0^2\,  p_{7,2} & p_0^3\,  p_{6,2}   & p_0^4 p_{5,2} & t^3\, p_0^2\,  p_{4,2} & 0	    & p_0^{k+6} \\
	\end{pmatrix}}\]

Also,

\[{\tiny T_1=\begin{pmatrix}
	0 &  &  &  &  &  &  &  &  \\
	0 & 0 &  &  &  &  &  &  &  \\
	0 & 0 & 0 &  &  &  &  &  &  \\
	0 & 0 & 0 & 0 &  &  &  &  &  \\
	0 & 0 & 0 & 0 & 0 &  &  &  &  \\
	0 & 0 & 0 & 0 & 0 & 0 &  &  &  \\
	0 & 0 & C^7_{3,4} & 0 & 0 & 0  & 0 &  &  \\
	0 & 0 & C^8_{3,4} & (C^7_{3,4} +C^8_{3,5}) p_0 & 0 & 0 & 0 & 0 &  \\
	0 & 0 & C^9_{3,4} & (C^8_{3,4} +C^9_{3,5}) p_0 & (C^7_{3,4} +C^8_{3,5}+C^9_{3,6} ) p_0^2 & 0 & 0 & 0 & 0 \\
	\end{pmatrix},}\]

\[{\tiny T_2=\begin{pmatrix}
	0 &  &  &  &  &  &  &  &  \\
	0 & 0 &  &  &  &  &  &  &  \\
	0 & 0 & 0 &  &  &  &  &  &  \\
	0 & 0 & 0 & 0 &  &  &  &  &  \\
	0 & 0 & 0 & 0 & 0 &  &  &  &  \\
	0 & 0 & 0 & 0 & 0 & 0 &  &  &  \\
	0 & 0 & 0 & 0 & 0 & 0  & 0 &  &  \\
	0 & 0 & C^8_{3,5} & 0 & 0 & 0 & 0 & 0 &  \\
	0 & 0 & C^9_{3,5} & (C^8_{3,5} +C^9_{3,6}) p_0 & 0 & 0 & 0 & 0 & 0 \\
	\end{pmatrix},}\]
	 
\[{\tiny T_3=\begin{pmatrix}
	0 &  &  &  &  &  &  &  &  \\
	0 & 0 &  &  &  &  &  &  &  \\
	0 & 0 & 0 &  &  &  &  &  &  \\
	0 & 0 & 0 & 0 &  &  &  &  &  \\
	0 & 0 & 0 & 0 & 0 &  &  &  &  \\
	0 & 0 & 0 & 0 & 0 &  &  &  &  \\
	0 & 0 & 0 & 0 & 0 & 0  &  &  &  \\
	0 & 0 & 0 & 0 & 0 & 0  &  &  &  \\
	0 & 0 & C^9_{3,6} & 0 & 0 & 0 & 0 & 0 & 0 \\
	\end{pmatrix},}\]

\[{\tiny T_4=\begin{pmatrix}
	0   \\
	0 & 0  \\
	p_0^{-(k+1)} & 0 & 0  \\
	0 & 0 & 0 & 0 \\
	0 & 0 & 0 & 0 & 0\\
	0 & 0 & 0 & 0 & 0 & 0\\
	0 & 0 & 0 & 0 &  C^7_{3,4}\,  & 0 & 0 \\
	0 & 0 & 0 & 0 & C^8_{3,4} \, &  (C^7_{3,4}+C^8_{3,5})\, p_0 & 0 & 0  \\
	0 & 0 & 0 & 0 & C^9_{3,4}  & (C^8_{3,4}+C^9_{3,5})\,  p_0^{} & \left(C^7_{3,4} +C^8_{3,5} +C^9_{3,6}\,  \right)\,  p_0^{2} & 0 & 0 \\
	\end{pmatrix}}\]

Finally, let $g_t\in GL_9$ given by:
\begin{equation}\label{matrizgcasoE2}
g_t=T_0+\tfrac{t-p_0^{k-1}}{k-1}  \left( p_{4,2}\, T_1+p_{5,2}\, T_2+ p_{6,2}\, T_3+  p_0^{k+1}\, T_4\right). 
\end{equation}
In the following sections, we  split the Lie algebras $\mathfrak{g}$ by the values ok $k$ that make $D$ a derivation of $\mathfrak{h}_1$. 
Then  we give specific  conditions on $p_0, p_{i,2}$  in order to  get $g$ as a  solution of \eqref{eqn:degeneration}.
\subsection{}
If $\mathfrak{g}\in \{ \mu_9^{26,\alpha, \beta},\mu_9^{27,\alpha},\mu_9^{28,\alpha},\mu_9^{29},\mu_9^{30},\mu_9^{31,\alpha},\mu_9^{32},\mu_9^{34}\}$ then $D$ as in \eqref{derivacionD1} is a derivation with $k=2,\, d_1=1$. In this case taking $p_0=t$ we have,
\begin{align*}
p_{4,2}&= -\frac{1}{2} \left(p_0-1\right) p_0^2 C^9{}_{2,7}\\  p_{5,2}&= -\frac{1}{3} \left(p_0-1\right) p_0^3 C^9{}_{2,6}\\  p_{6,2}&= \frac{1}{8} \left(p_0-1\right) p_0^3 \left(p_0 \left(C^9_{2,7}\right)^2-2 p_0 C^9{}_{2,5}-\left(C^9_{2,7}\right)^2\right)\\  p_{7,2}&= \frac{1}{30} \left(p_0-1\right) p_0^4 \left(-6 p_0 C^9{}_{2,4}+5 p_0 C^9{}_{2,6} C^9{}_{2,7}-5 C^9{}_{2,6} C^9{}_{2,7}\right)\\  
p_{8,2}&= -\frac{1}{144} \left(p_0-1\right) p_0^4 \left(3 p_0^2 \left(C^9_{2,7}\right)^3-6 p_0 \left(C^9_{2,7}\right)^3-18 p_0^2 C^9{}_{2,5} C^9{}_{2,7} +\right.\\
&\quad \left.+18 p_0 C^9{}_{2,5} C^9{}_{2,7}-8 p_0^2 \left(C^9_{2,6}\right)^2+8 p_0 \left(C^9_{2,6}\right)^2+24 p_0^2 C^9{}_{2,3}+3 \left(C^9_{2,7}\right)^3\right)
\end{align*}

\subsection{}
If $\mathfrak{g}\in \{ \mu_9^{13,\alpha}, \mu_9^{14,\alpha}, \mu_9^{15}\}$ then $D$ as in \eqref{derivacionD1} is a derivation with $k=3,\, d_1=1$. In this case taking $p_0=t$ we have,
\begin{align*}\label{casoD3}
p_{4,2}&= 0\\  p_{5,2}&= -\frac{1}{3} \left(p_0-1\right) p_0^4 C^6_{2,3}\\  p_{6,2}&= -\frac{1}{4} \left(p_0-1\right) p_0^5 C^9_{2,5}\\  p_{7,2}&= -\frac{1}{5} \left(p_0-1\right) p_0^6 C^9_{2,4}\\  p_{8,2}&= \frac{1}{36} \left(p_0-1\right) p_0^6 \left(2 p_0 \left(C^9_{2,6}\right)^2-2 \left(C^9_{2,6}\right)^2+3 C^9_{2,5}\right)
\end{align*}

\subsection{}
If $\mathfrak{g}\in \{\mu_9^{2,\alpha}, \mu_9^{5,\alpha}\}$ then $D$ as in \eqref{derivacionD1} is a derivation with $k=3,\,  d_1=\tfrac{C^9_{4,5} C^5_{2,3}}{C^9_{3,4}}$. In this case taking $p_0=t$ we have,
{\small 
	\begin{align*}
	p_{4,2}&= \frac{(1-t) t^3 C^9_{3,4}}{2 C^9_{4,5}}\\ 
	p_{5,2}&=  \frac{(1-t) t^4 C^6_{2,3} C^9_{3,4}}{3 C^5{}_{2,3} C^9{}_{4,5}}\\
	p_{6,2}&= \frac{(t-1) t^3 \left(C_{3,4}^9\right)^2 }{8 C^5_{2,3} \left(C_{4,5}^9\right)^2}\left(t^2 C^5_{2,3}-t C^5_{2,3}-2 t C^7_{3,4}-C^5_{2,3}+C^7_{2,5}+C^7_{3,4}\right)\\ 
	p_{7,2}&=  \frac{(t-1) t^3 \left(C_{3,4}^9\right)^2}{30 \left(C_{2,3}^5\right)^2 \left(C^9_{4,5}\right)^2} \left(C^5_{2,3} \left(\left(5 t^3-5 t^2-2 t-3\right) C^6_{2,3}+3 C^8_{2,5}\right)+\right.\\
	&\quad \left. +t C^6_{2,3} \left(3 (t-1) C^7_{3,4}+2 (1-2 t) C^8_{3,5}+2 C^8_{2,6}\right)\right)\\ 
	p_{8,2} &= \tfrac{1}{(-1 + t) C^8_{2,3} }\left( C^9_{3,4} \left(t^3 p_{4,2}-p_{3,2} p_{4,3}\right)+C^9_{3,5} \left(t^3 p_{5,2}-p_{3,2} p_{5,3}\right)+C^9_{3,6} \left(t^3 p_{6,2}-p_{3,2} p_{6,3}\right)+\right.\\
	& \quad \left. +t p_{4,3} C^8_{2,3}-t p_{4,3} C^9_{2,4}-t p_{5,3} C^9_{2,5}-t p_{6,3} C^9_{2,6}-t\,  p_{7,3} C^9_{2,7}+C^5_{2,3} \left(t \left(p_{7,3}-p_{9,5}\right)+p_{9,5}\right)\right.\\
	& \quad \left. +C^6_{2,3} \left(t \left(p_{6,3}-p_{9,6}\right)+p_{9,6}\right)+ C^7_{2,3} \left(t \left(p_{5,3}-p_{9,7}\right)+p_{9,7}\right)+\left(p_{4,3} p_{5,2}-p_{4,2} p_{5,3}\right) C^9_{4,5}+\right.\\
	& \quad \left.-(t-1) t^9 C^9_{2,3}-6 t \, d_{1,1} p_{9,3} \right)
	\end{align*}
}

\subsection{}
If $\mathfrak{g}\in\{\mu_9^{17,\alpha, 0},\mu_9^{21,\alpha} \}$ then $D$ as in \eqref{derivacionD1} is a derivation with $k=4,\, d_1=1$. In this case taking $p_0^3=t$ we have,
\begin{align*}
p_{4,2}&=p_{7,1}=0\\
p_{5,2}&= \frac{1}{3} p_0^3 \left(p_0^3-1\right) \left(-C^9{}_{2,6}-2\right)\\  p_{6,2}&= -\frac{1}{4} p_0^4 \left(p_0^3-1\right) C^9{}_{2,5}\\  p_{8,2}&= \frac{1}{18} p_0^3 \left(p_0^3-1\right) \left(C^9{}_{2,6}+2\right) \left(p_0^3 C^9{}_{2,6}-C^9{}_{2,6}+2 p_0^3-5\right)
\end{align*}

\begin{remark}
	This is not a solution for $\mu_9^{17,\alpha, \beta}$ if $\beta \neq  0$ since $D$ is not a derivation. Furthermore, all the derivations of $\mathfrak{h}_1$ are nilpotent.
\end{remark}
\subsection{}
If $\mathfrak{g}\in\{\mu_9^{18,\alpha, \beta},\mu_9^{19,\alpha},\mu_9^{23,\alpha},\mu_9^{24,\alpha}\} $ then $D$ as in \eqref{derivacionD1} is a derivation with $k=5,\, d_1=1$. In this case taking $p_0=t$ we have,
\begin{align*}
p_{4,2}&= -\frac{1}{2} \left(p_0-1\right) p_0^5 C^9{}_{2,7}\\  p_{5,2}&= -\frac{1}{3} \left(p_0-1\right) p_0^6 C^9{}_{2,6}\\  p_{6,2}&= \frac{1}{8} \left(p_0-1\right) p_0^6 \left(p_0 \left(C^9_{2,7}\right)^2-2 p_0 C^9{}_{20,4}-\left(C^9_{2,7}\right)^2-2 p_0\right)\\  p_{7,2}&= \frac{1}{30} \left(p_0-1\right) p_0^7 \left(-6 p_0 C^9{}_{2,4}+5 p_0 C^9{}_{2,6} C^9{}_{2,7}-5 C^9{}_{2,6} C^9{}_{2,7}\right)\\  
p_{8,2}&= -\frac{1}{144} \left(p_0-1\right) p_0^7 \left(3 p_0^2 \left(C^9_{2,7}\right)^3-6 p_0 \left(C^9_{2,7}\right)^3-18 p_0^2 C^9{}_{2,7}+42 p_0 C^9{}_{2,7}+\right.\\
&\quad \left.-18 p_0^2 C^9{}_{20,4} C^9{}_{2,7}+18 p_0 C^9{}_{20,4} C^9{}_{2,7}-8 p_0^2 \left(C^9_{2,6}\right)^2+8 p_0 \left(C^9_{2,6}\right)^2+3 \left(C^9_{2,7}\right)^3\right)
\end{align*}

\subsection{Other  cases }
We can solve two more special cases where $\mathfrak{h}=\mathfrak{h}_1$, but in this case \eqref{derivacionD1} is not a derivation; hence we have to construct the solution $g_t$ for these cases.   Let $\mathfrak{g}=\mu_9^{6,\alpha}$ then $D$ as in \eqref{derivacionD1} is a derivation with $t=p_0^2$.

{\tiny 
	\[D=\begin{pmatrix}	1 & 0 & 0 & 0 & 0 & 0 & 0 & 0 \\
	0 & 3  & 0 & 0 & 0 & 0 & 0 & 0 \\
	0 & 0 & 4  & 0 & 0 & 0 & 0 & 0 \\
	0 & 0 & 0 & 5  & 0 & 0 & 0 & 0 \\
	0 & 0 & 0 & 0 & 6  & 0 & 0 & 0 \\
	0 & 0 & 0 & 0 & 0 & 7  & 0 & 0 \\
	0 & 0 & 0 & 0 & 0 & 0 & 8  & 0 \\
	0 & 0 & 0 & 0 & 0 & 0 & 0 & 9 \\
	\end{pmatrix}\]
and	
	\[ g_t=\begin{pmatrix}
	p_0 & 0 & 0 & 0 & 0 & 0 & 0 & 0 & 0 \\
	0 & p_0^2 & 0 & 0 & 0 & 0 & 0 & 0 & 0 \\
	0 & 0 & p_0^3 & 0 & 0 & 0 & 0 & 0 & 0 \\
	0 & 0 & 0 & p_0^4 & 0 & 0 & 0 & 0 & 0 \\
	0 & 0 & 0 & 0 & p_0^5 & 0 & 0 & 0 & 0 \\
	0 & 0 & 0 & 0 & 0 & p_0^6 & 0 & 0 & 0 \\
	0 & 0 & 0 & 0 & 0 & 0 & p_0^7 & 0 & 0 \\
	0 & -\frac{ p_0^2 (p_0^4-1)}{6 } & 0 & 0 & 0 & 0 & 0 & p_0^8 & 0 \\
	0 &0 & -\frac{p_0^3 (p_0^4-1)}{6 } & 0 & 0 & 0 & 0 & 0 & p_0^9 \\
	\end{pmatrix}\]
}

Let $\mathfrak{g}=\mu_9^{7,\alpha}$. Consider the derivation on $\mathfrak{h}_1$:
{\tiny
	\[
	D=\begin{pmatrix}	
	1 & 0 & 0 & 0 & 0 & 0 & 0 & 0 \\
	0 & 3  & 0 & 0 & 0 & 0 & 0 & 0 \\
	0 & 0 & 4  & 0 & 0 & 0 & 0 & 0 \\
	0 & -\tfrac{\alpha}{3} & 0 & 5  & 0 & 0 & 0 & 0 \\
	0 & -1 & -\tfrac{\alpha}{3} & 0 & 6  & 0 & 0 & 0 \\
	0 & -\tfrac{\alpha^2}{18} & -1 & -\tfrac{\alpha}{3} & 0 & 7  & 0 & 0 \\
	0 & 0 & -\tfrac{\alpha^2}{18} & -1 & -\tfrac{\alpha}{3} & 0 & 8  & 0 \\
	0 & -\tfrac{6\alpha (  3  - 5 \alpha )}{324} & 0 & -\tfrac{\alpha^2}{18} & -1 &-\tfrac{\alpha}{3} & 0 & 9 \\
	\end{pmatrix}\]
}

taking $t=p_0^2$ we can find a solution of \eqref{eqn:linear-deformation} given by,

{\tiny	\[ g_t=\begin{pmatrix}
	p_0 & 0 & 0 & 0 & 0 & 0 & 0 & 0 & 0 \\
	0 & p_0^2 & 0 & 0 & 0 & 0 & 0 & 0 & 0 \\
	0 & 0 & p_0^3 & 0 & 0 & 0 & 0 & 0 & 0 \\
	0 & p_{4,2} & 0 & p_0^4 & 0 & 0 & 0 & 0 & 0 \\
	0 & p_{5,2} & p_{5,3} & 0 & p_0^5 & 0 & 0 & 0 & 0 \\
	0 & p_{6,2} & p_{6,3} & p_{6,4} & 0 & p_0^6 & 0 & 0 & 0 \\
	0 &p_{7,2} & p_{7,3} & p_{7,4} & p_{7,5} & 0 & p_0^7 & 0 & 0 \\
	0 & p_{8,2} & p_{8,3} & p_{8,4}  & p_{8,5} & p_{8,6} & 0 & p_0^8 & 0 \\
	0 & 0 & p_{9,3} & p_{9,4} & p_{9,5}& p_{9,6} & p_{9,7} & 0 & p_0^9 \\
	\end{pmatrix} \]}

where
\begin{align*}
p_{4,2}&=-\frac{1}{6} a (p_0-1) p_0^2 (p_0+1)\\
p_{5,2}&=-\frac{1}{3} (p_0-1) p_0^3 (p_0+1)\\
p_{5,3}&=-\frac{1}{6} a (p_0-1) p_0^3 (p_0+1)\\
p_{6,2}&= -\frac{1}{36} a (a+3) (p_0-1) p_0^2 (p_0+1)\\
p_{6,3}&=-\frac{1}{3} (p_0-1) p_0^4 (p_0+1)\\
p_{6,4}&=-\frac{1}{6} a (p_0-1) p_0^4 (p_0+1)\\
p_{7,2}&=\frac{1}{90} (p_0-1) p_0^3 (p_0+1) \left(5 a p_0^2-5 a-18\right)\\
p_{7,3}&= -\frac{1}{36} a (a+3) (p_0-1) p_0^3 (p_0+1)\\
p_{7,4}&=-\frac{1}{3} (p_0-1) p_0^5 (p_0+1)\\
p_{7,5}&= -\frac{1}{6} a (p_0-1) p_0^5 (p_0+1)\\
p_{8,2} &= \frac{1}{648} (p_0-1)^2 p_0^2 (p_0+1)^2 \left(p_0^2 a^3+2 a^3+10 p_0^2 a^2+14 a^2-6 p_0^2 a-12 a+36 p_0^2\right)\\
p_{8,3}&=\frac{1}{90} (p_0-1) p_0^4 (p_0+1) \left(5 a p_0^2-5 a-18\right)\\
p_{8,4}&= -\frac{1}{36} a (a+3) (p_0-1) p_0^4 (p_0+1)\\
p_{8,5}&= -\frac{1}{3} (p_0-1) p_0^6 (p_0+1)\\
p_{8,6}&= -\frac{1}{6} a (p_0-1) p_0^6 (p_0+1)\\
p_{9,3}&=\frac{1}{648} (p_0-1)^2 p_0^3 (p_0+1)^2 \left(p_0^2 a^3+2 a^3+10 p_0^2 a^2+14 a^2-6 p_0^2 a-12 a+36 p_0^2\right)\\
\end{align*}
\begin{align*}
p_{9,4}&= \frac{1}{90} (p_0-1) p_0^5 (p_0+1) \left(5 a p_0^2-5 a-18\right)\\
p_{9,5}&= -\frac{1}{36} a (a+3) (p_0-1) p_0^5 (p_0+1)\\
p_{9,6}&= -\frac{1}{3} (p_0-1) p_0^7 (p_0+1)\\
p_{9,7}&= -\frac{1}{6} a (p_0-1) p_0^7 (p_0+1)
\end{align*}

\section{Case $\mathfrak{h}=\mathfrak{h}_2$}

Let $\mathfrak{h}_2$ the Lie  ideal with basis $\mathcal{B}_2=\{X_2, X_3, X_4, \ldots, X_9\}$. Suppose that there exist a derivation  of $\mathfrak{h}_2$  with transformation matrix $D$ relative to the basis $\mathcal{B}_2$ given by, 

\begin{equation}\label{D caso h1}
{\tiny 
	D=\begin{pmatrix}
	d_{1,1} & 0 & 0 & 0 & 0 & 0 & 0 & 0 \\
	d_{2,1} & d_{2,2}  & 0 & 0 & 0 & 0 & 0 & 0 \\
	d_{3,1} & d_{3,2}  & d_{3,3}  & 0 & 0 & 0 & 0 & 0 \\
	d_{4,1} & d_{4,2}  & d_{4,3} &  d_{4,4}  & 0 & 0 & 0 & 0 \\
	d_{5,1} & d_{5,2}  & d_{5,3} & d_{5,4} & d_{5,5}  & 0 & 0 & 0 \\
	d_{6,1} & d_{6,2}  & d_{6,3}  & d_{6,4} & d_{6,5} &  d_{6,6}  & 0 & 0 \\
	d_{7,1} & d_{7,2}  & d_{7,3}  & d_{7,4} & d_{7,5} & d_{7,6} & d_{7,7}  & 0 \\
	d_{8,1} & d_{8,2}  & d_{8,3}  & d_{8,4} & d_{8,5} & d_{8,6} & d_{8,7} &  d_{8,8}  \\
	\end{pmatrix}}
\end{equation}

such that  $d_{i,i}\neq d_{j,j}$, \,  if $i \neq j$.  If such $D$ exists,  consider the following linear map on $\mathfrak{g}$ with transformation matrix relative to the basis $\mathcal{B}$ given by, 

\begin{equation}\label{g caso h1}
{\tiny g_t=\begin{pmatrix}
	t &  &  &  &  &  &  &  & \\[-0.1cm]
	0 & t^{d_{1,1}} &  &  &  &  &  &  & \\[-0.1cm]
	0 & p_{3,2} & t^{d_{2,2}} &  &  &  &  &  & \\[-0.1cm]
	0 & p_{4,2} & p_{4,3} & t^{d_{3,3}} &  &  &  &  & \\[-0.1cm]
	0 & p_{5,2} & p_{5,3} & p_{5,4} & t^{d_{4,4}} &  &  &  & \\[-0.1cm]
	0 & p_{6,2} & p_{6,3} & p_{6,4} & p_{6,5} &t^{d_{5,5}} &  &  & \\[-0.1cm]
	0 & p_{7,2} & p_{7,3} & p_{7,4} & p_{7,5} & p_{7,6} & t^{d_{6,6}} &  &  \\[-0.1cm]
	0 & p_{8,2} & p_{8,3} & p_{8,4} & p_{8,5} & p_{8,6} & p_{8,7} & t^{d_{7,7}} &  \\[-0.1cm]
	0 & p_{9,2} & p_{9,3} & p_{9,4} & p_{9,5} & p_{9,6} & p_{9,7} & p_{9,8} & t^{d_{8,8}} \\       
	\end{pmatrix}}\end{equation}

Clearly $g_t\in GL_9$ for all $t\neq 0$. We would like to find a solution of \eqref{eqn:degeneration} with these $D$ and $g_t$. Thus, we have to determine each value of $p_{i,j}$. Write \[\mu_1(g_t(X_i),g_t(X_j))- g_t \cdot \mu_t(X_i, X_j)=\sum_{l=1}^9 q(i,j,l) X_l \].

By the definitions of $g_t$ and $D$, the subspace $\langle X_i, X_{i+1},\ldots , X_9 \rangle$ is $g_t$- invariant and also $D$-invariant if $i\ge 2$. Thus by \eqref{Corchetes-Base}   $q(i,j,l)=0$ if $i+j>9$ or $l<i+j$.

For $1< s<r$ we have,
\begin{align*}
q(1,r,s)&=-t(d_{r-1,r-1}-d_{s-1,s-1})p_{r,s}+ t\, d_{r-1, s-1}(p_{r,r}-p_{s,s})+(p_{r, s+1}- t\,  p_{r-1,s})+\\
&\quad + \sum_{i=s+1}^{r-1} t\, (d_{i-1, s-1}\,  p_{r,i} - d_{r-1, i-1}\,  p_{i,s})\end{align*}

In particular, 
\[q(1,s+1,s)=-t(d_{s+1,s+1}-d_{s,s})p_{s+1,s}+ t\, d_{s, s-1}(p_{s+1,s+1}-p_{s,s})+(p_{s+1, s+1}- t\,  p_{s,s})\]

Since we know the values of the derivation $D$ and   the diagonal of $g_t$,  and since $d_{s,s}\neq d_{s+1,s+1}$, if  $q(1,s+1,s)=0$ we have:

\begin{equation}\label{primera diagonal}
\begin{aligned}
p_{s+1,s}&=Y(s+1, s)\Big( t\, d_{s, s-1}(p_{s+1,s+1}-p_{s,s}) +(p_{s+1, s+1}- t\,  p_{s,s})\Big)\\
&= p_{s+1,s+1} (1+t\, d_{s, s-1}) Y(s+1, s) +p_{s,s}\,  t(1+ \, d_{s, s-1}) Y(s, s+1)\\
\end{aligned}
\end{equation}

where 
\begin{equation}
Y(i,j)=\dfrac{1}{t(d_{i-1,i-1}-d_{j-1,j-1}  )   }.\\
\end{equation}

Recursively, if $p_{i,j}$ have been defined for all $i,j$ with $i-j<r-s$ we have,
\begin{equation}\label{recurrencia}
\begin{aligned}
p_{r,s}&= t \, d_{r-1, s-1}(p_{r,r}-p_{s,s})Y(r,s)+(p_{r, s+1}- t\,  p_{r-1,s})Y(r,s)+\\
&\quad + \sum_{i=s+1}^{r-1} t\, (d_{i-1, s-1}\,  p_{r,i} - d_{r-1, i-1}\,  p_{i,s})Y(r,s)
\end{aligned}
\end{equation}

In order to solve this recurrence explicitly we need to introduce some notation. \bigskip

Let $k\ge 2$ and positive integers $r>s$. Let $I_k$ be the set of  decreasing integer sequences of length $k$ starting in $r$ and finishing in $s$ .i.e; 
\[I_k(r,s)=\{r=a_1>a_2>\ldots >a_k=s\mid  a_i\in \mathbb{Z} \}\]
For $x,y,z$ such that $x>y$  we define:
\[f_z(x,y)=\begin{cases}
t\,  d_{x-1,y-1} &\quad \text{if \, } x-y>1\\
t(1+ d_{x-1,y-1}) &\quad \text{if \, } z\le y=x-1\\
(1+t\,  d_{x-1,y-1}) &\quad \text{if \, }z>y=x-1\\
\end{cases}\]
For simplification, we denote
\[Y_k(a_i,a_j)=\begin{cases}
Y(a_i,a_j)&\quad \text{if \, } i\neq j\\
Y(a_i,a_k)&\quad \text{if \, } i=j
\end{cases}\]
For $\vec{a}=(a_1, \ldots, a_i,\ldots,a_k)\in I_k(s,r)$ we denote $l(\vec{a})=k$. For $1\le i\le k$ let:
\[F_i(\vec{a})=p_{a_i,a_i}\left( \prod_{j=1}^{k-1}f_{a_i}(a_j,a_{j+1}) Y_k(a_{i}, a_{j})\right) \]

For  instance, for $\vec{a}=(8,6,5,4,2)\in I_5(8,2)$ 
\begin{align*}
F_3(\vec{a})&=p_{5,5}\Big(f_5(8,6)f_5(6,5)f_5(5,4)f_5(4,2)Y(5,8)Y(5,6)Y(5,2)Y(5,4)\Big)\\
&=p_{5,5}\frac{(t \, p_{7, 5}) \,\big( t(1+  \, p_{5,4})\big)(1+t \, p_{4,3})( t \, p_{3,1})}{t(d_{4,4}-d_{7,7})t(d_{4,4}-d_{5,5})t(d_{4,4}-d_{1,1})t(d_{4,4}-d_{3,3})}\\
&=p_{5,5}\frac{ \, p_{7, 5} \, (1+  \, p_{5,4})(1+t \, p_{4,3}) p_{3,1}}{t (d_{4,4}-d_{7,7})(d_{4,4}-d_{5,5})(d_{4,4}-d_{1,1})(d_{4,4}-d_{3,3})}\\
\end{align*}

and

\begin{align*}
F_2(\vec{a})&=p_{4,4}\frac{ \, p_{7, 5} \, (1+  \, p_{5,4})(1+ \, p_{4,3}) p_{3,1}}{ (d_{4,4}-d_{7,7})(d_{3,3}-d_{5,5})(d_{3,3}-d_{4,4})(d_{3,3}-d_{1,1})}\\
\end{align*}

\begin{proposition}
	Let $D$ be a derivation of $\mathfrak{h}_2$ as in  \eqref{D caso h1}  and  $g_t\in GL_9$ as in \eqref{g caso h1}. For $r>s$ let   $I_{r,s}=\coprod\limits_{k=2}^{r-s+1} I_k(r,s)$. If $q(1,r,s)=0$ then:
	\begin{equation}\label{coeficiente r,s}
	p_{r,s}=\sum_{\vec{a}\in I_{r,s}}\sum_{i=1}^{l(\vec{a})} F_i(\vec{a})
	\end{equation}
	
\end{proposition}
\begin{proof}
	By induction on $n=r-s$. The case $n=1$ follows from
	\eqref{primera diagonal}. By induction, suppose that  the formula is valid for all $p_{i,j}$ such that $i-j=n$. Let $r,s$ such that $r-s=n+1$. We can rewrite  \eqref{recurrencia} in the form,
	
	\begin{equation}\label{pr}
	\begin{aligned}
	p_{r,s}&= t \, d_{r-1, s-1}p_{r,r} Y(r,s)  + t \, d_{r-1, s-1} p_{s,s} Y(s,r)+\\
	&\quad +(1+t\,  d_{s,s-1})p_{r,s+1} Y(r,s) + t\,  d_{r-1, s}\,  p_{s+1,s} Y(s,r)+\\
	&\quad + t\, d_{r-2, s-1}\,  p_{r,r-1} Y(r,s) + \, t\,  (1+d_{r-1, r-2})p_{r-1,s} Y(s,r)  +  \\
	&\quad + \sum_{l=s+2}^{r-2} t\, d_{l-1, s-1}\,  p_{r,l}\, Y(r,s)+  t d_{r-1, l-1}\,  p_{l,s} Y(s,r)\\
	&=  p_{r,r}\, f_r(r, s) Y(r,s)  + p_{s,s} f_s(r, s) Y(s,r) +\\
	&\quad + p_{r,s+1}  f_{r}(s+1, s) Y(r,s) +p_{s+1,s} \, f_s(r, s+1) Y(s,r) +\\
	&\quad + p_{r,r-1}f_{r}(r-1,s)Y(r,s) + p_{r-1,s}f_{s}(r,r-1)Y(s,r) +\\
	&\quad + \sum_{l=s+2}^{r-2}p_{r,l}f_{r}(l,s)Y(r,s) + p_{l,s}f_{s}(r,l)Y(s,r)\\
	&= \sum_{l=s}^{r}\big( p_{r,l}f_{r}(l,s)Y(r,s) + p_{l,s}f_{s}(r,l)Y(s,r)\big)
	\end{aligned}
	\end{equation}

	Let $\vec{a}=(a_1,a_2, \ldots, a_k)\in I_{r,s}$. Then $ Y_k(a_{i}, a_i) Y_k(a_{i}, a_1)=Y(a_{i}, s) Y(a_{i}, r)$. Moreover the $Y_{i,j}$ satisfy
	\[Y_{x,y} Y_{x,z}= Y_{x,y} Y_{y,z} +Y_{x,z} Y_{z,y} \] 
	
	Thus 
	$ Y(a_{i}, s) Y(a_{i}, r)= Y(a_{i}, s) Y(s, r)+Y(a_{i}, r) Y(r, s)$. Note also that $f_{a_i}(r,a_{2})=f_{s}(r,a_{2})$ and $f_{a_i}(a_{k-1}, s)=f_{r}(a_{k-1}, s)$ for $1<i<k$. These identities for $1<i<k$  implies

	\begin{align*}  F_i(\vec{a})&= p_{a_i,a_i}\left( \prod_{j=1}^{k-1}f_{a_i}(a_j,a_{j+1}) Y_k(a_{i}, a_{j})\right)\\
	&=f_{s}(r,a_{2}) \, \,  p_{a_i,a_i}\left( \prod_{j=2}^{k-1}f_{a_i}(a_j,a_{j+1}) Y_k(a_{i}, a_{j})\right) Y( s,r)+\\
	&\quad  + f_{r}(a_{k-1},s) \, \,  p_{a_i,a_i}\left( \prod_{j=1}^{k-2}f_{a_i}(a_j,a_{j+1}) Y_k(a_{i}, a_{j})\right) Y(r,s)\\
	&= f_{s}(r,a_{2}) Y( s,r) \, \, F_{i-1}(a_2,a_3, \ldots s) + f_{r}(a_{k-1},s) Y( r,s)  F_{i}(r,a_2, \ldots a_{k-1}) 
	\end{align*}
	
	Hence:
	
	\begin{align*}
	\sum_{i=1}^k F_i(\vec{a})&=F_1(\vec{a})+\sum_{i=2}^{k-1} F_i(\vec{a})+F_k(\vec{a})  \\
	&= f_{r}(a_{k-1},s) Y(r,s) \,  F_{1}(r,a_2, \ldots a_{k-1})+\\
	&\quad  +f_{s}(r,a_{2}) Y(s,r) \sum_{i=2}^{k-1} F_{i-1}(a_2,a_3, \ldots s) +\\
	&\quad  + f_{r}(a_{k-1},s) Y(r,s) \sum_{i=2}^{k-1} F_{i}(r,a_2, \ldots a_{k-1})+ \\
	&\quad + f_{s}(r,a_{2}) Y(s,t)  F_{k-1}(a_2, \ldots s) \\
	&=f_{r}(a_{k-1},s) Y(s,r) \, \sum_{i=1}^{k-1}  F_{i}(a_2,a_3, \ldots s)+\\
	&\quad  + f_{s}(r,a_{2}) Y(r,s)  \sum_{i=1}^{k-1} F_{i}(r,a_2, \ldots a_{k-1})
	\end{align*}
	Summing  over all the decreasing sequences $\vec{a}\in I_{r.s}$, the first summand   by induction hypothesis will be equal to $f_{r}(a_{k-1},s) Y(s,r) \, p_{r,a_2}$ and the second summand to 
	$f_{s}(r,a_{2}) Y(r,s)  \, p_{a_{k-1},s}$. Thus the right hand side of \eqref{coeficiente r,s} is 
	\[\sum_{\vec{a}\in I_{r.s}} f_{r}(a_{k-1},s) Y(s,r) \, p_{r,a_2}+f_{s}(r,a_{2}) Y(r,s)  \, p_{a_{k-1},s}\]
	
	but this is coincides with the last equality  in \eqref{pr} which proves the inductive the inductive step.
\end{proof}

Now,  is enough to find $D$ such that $q(2,i,l), q(3,i,l), q(4,i,l)$ are zero. This can be done computationally using a software. In the following subsections we give the list of each algebra and each derivation $D$. By the last proposition, the matrix $g_y$ is determined in terms of the coefficients of $D$.

\subsection{Case $\g=\mu_9^{1,\alpha,\beta}$}

For $\alpha\notin\{-1,0,1, \tfrac{1}{2}\}$

{\tiny\[D=\begin{pmatrix}
	2 & 0 & 0 & 0 & 0 & 0 & 0 & 0 \\[-0.1cm]
	0 & 3 & 0 & 0 & 0 & 0 & 0 & 0 \\[-0.1cm]
	0 & p_{3,2} & 4 & 0 & 0 & 0 & 0 & 0 \\
	0& 0 & p_{4,3} & 5 & 0 & 0 & 0 & 0 \\[-0.1cm]
	0 & 0 & p_{5,3} & p_{5,4} & 6 & 0 & 0 & 0 \\[-0.1cm]
	0 & 0 & 0 & 0 & p_{6,5} & 7 & 0 & 0 \\[-0.1cm]
	0 & 0 & 0 & 0 & p_{7,5} &  p_{7,6} & 8 & 0 \\[-0.1cm]
	0 & 0 & 0 & 0 & p_{8,5} & p_{8,6} & p_{8,7} & 9 \\[-0.1cm]
	\end{pmatrix}\]}

where

\begin{align*}
p_{3,2} =& \frac{2(\alpha-1)}{3 \alpha^2 (2 \alpha-1)}\\
p_{4,3} =& \frac{(\alpha+2)}{3\,  \alpha^2 (1-\alpha)}\\
p_{5,3} =& \frac{(\alpha+2) (\alpha\,  \beta-1)}{(\alpha-1)^2 \alpha^2}\\
p_{5,4} =& \frac{2(\alpha-1)}{3 \alpha^2 (2 \alpha-1)}\\
p_{6,5} =& \frac{\alpha+2}{3\,  \alpha^2}\\
p_{7,5} =& -\frac{(\alpha+2) (6 \alpha^2 \beta-3 \alpha \beta-7 \alpha+4)}{3(\alpha-1)^2 \alpha^2}\\
p_{7,6} =& \frac{3 \alpha^2-2 \alpha+2}{3 \alpha^2(1-\alpha)}\\
p_{8,5} =&  -\frac{(\alpha+2) (5 \alpha \beta+\beta-6)}{3 (\alpha-1)^2 \alpha^2}\\
p_{8,6} =& \frac{12 \alpha^3 \beta-6 \alpha^2 \beta-6 \alpha^2+\alpha+2}{3 \alpha^2 (2 \alpha-1)(1-\alpha)}\\
p_{8,7} =&  \frac{6 \alpha^2-5 \alpha+2}{1-2 \alpha^2 (2 \alpha  )}
\end{align*}

\subsection{$\g=	\mu_9^{4,\alpha,\beta}$ }\ 

{\tiny \[D=\begin{pmatrix}
	2 & 0 & 0 & 0 & 0 & 0 & 0 & 0 \\[-0.1cm]
	0 & 3 & 0 & 0 & 0 & 0 & 0 & 0 \\[-0.1cm]
	0 & \frac{\alpha}{56} & 4 & 0 & 0 & 0 & 0 & 0 \\[-0.1cm]
	0 & \frac{1}{24} \left(\alpha^2-4 \beta\right) & -\frac{\alpha}{24} & 5 & 0 & 0 & 0 & 0 \\[-0.1cm]
	0 & 0 & 0 & \frac{\alpha}{56} & 6 & 0 & 0 & 0 \\[-0.1cm]
	0 & -\frac{4}{3} & 0 & \frac{1}{8} \left(4 \beta-\alpha^2\right) & \frac{\alpha}{8} & 7 & 0 & 0 \\[-0.1cm]
	0 & 0 & 0 & -\frac{1}{24} \alpha \left( \alpha^2 -4 \beta \right) & \frac{\alpha^2}{24} & -\frac{1}{24} (7 \alpha) & 8 & 0 \\[-0.1cm]
	0 & 0 & 0 & -\frac{1}{24} \left(\alpha^2 -4 \beta \right) \beta & \frac{\alpha \beta}{24} & \frac{1}{84} \left(-3 \alpha^2-14 \beta\right) & -\frac{1}{56} (13 \alpha) & 9 \\[-0.1cm]
	\end{pmatrix}\]}

\subsection{$\g=	\mu_9^{8}$ }\

{\tiny\[D=\begin{pmatrix}
	2 &  &  &  &  &  &  & \\[-0.1cm]
	0 & 3 &  &  &  &  &  & \\[-0.1cm]
	0 & 0 & 4 &  &  &  &  & \\[-0.1cm]
	0 & 0& 0 & 5 &  &  &  &  & \\[-0.1cm]
	0 & 0 & 0 & 0 & 6 &  &  & \\[-0.1cm]
	0 & 0 & 0 & 0 & 0 & 7 &  & \\[-0.1cm]
	0 & 0 & 0 & 0 & 0 & 0 & 8 & \\[-0.1cm]
	0 & 0 & 0 & 0 & 0 & -2 & 0 & 9\\       
	\end{pmatrix}\]}

\subsection{$\g=\mu_9^{10,\alpha,\beta}$ }\ 

{	\tiny\[D=\begin{pmatrix}
	3 &  &  &  &  &  &  & \\[-0.1cm]
	0 & 4 &  &  &  &  &  & \\[-0.1cm]
	0 & \frac{\alpha^2-2 \alpha+\beta}{\alpha} & 5 &  &  &  &  & \\[-0.1cm]
	0 & 0& 0 & 6 &  &  &  &  & \\[-0.1cm]
	0 & 0 & 0 & \frac{\alpha^2-3 \alpha+\beta}{\alpha} & 7 &  &  & \\[-0.1cm]
	0 & 0 & 0 & -\frac{(\alpha^3-5\alpha^2-\alpha\beta+6\alpha-2\beta}{\alpha} & \alpha -2& 8 &  & \\[-0.1cm]
	0 & 0 & 0 & 0 & \frac{\beta(\alpha^2-2\alpha+2\beta}{\alpha^2} & -\frac{\beta}{\alpha}  & 9 & \\[-0.1cm]
	0 & 0 & 0 & 0 & 0 & 0 & 0 & 11\\       
	\end{pmatrix}\]}

\subsection{$\g=\mu_9^{11,\alpha,\beta}$ }\ 
{\tiny\[D=\begin{pmatrix}
	3 &  &  &  &  &  &  & \\[-0.1cm]
	0 & 4 &  &  &  &  &  & \\[-0.1cm]
	0 & \alpha-1 & 5 &  &  &  &  & \\[-0.1cm]
	0 & 0& 0 & 6 &  &  &  &  & \\[-0.1cm]
	0 & 0 & 0 & \alpha-2 & 7 &  &  & \\[-0.1cm]
	0 & 0 & 0 & 3\alpha-2\beta -2 & -1& 8 &  & \\[-0.1cm]
	0 & 0 & 0 & 0 & 2\alpha^2-\alpha  & -\alpha  & 9 & \\[-0.1cm]
	0 & 0 & 0 & 0 & -2\alpha^2\beta & 2\alpha \beta & -2\beta & 11\\       
	\end{pmatrix}\]}

\subsection{$\g=\mu_9^{12,\alpha,\beta}$ }\ 
{\tiny\[D=\begin{pmatrix}
	3 &  &  &  &  &  &  & \\[-0.1cm]
	0 & 4 &  &  &  &  &  & \\[-0.1cm]
	0 & -3\alpha-\tfrac{7}{3} & 5 &  &  &  &  & \\[-0.1cm]
	0 & 0& 0 & 6 &  &  &  &  & \\[-0.1cm]
	0 & 0 & 0 & -3\alpha -\tfrac{10}{3} & 7 &  &  & \\[-0.1cm]
	0 & 0 & 0 & -5\alpha-\tfrac{70}{9} +\tfrac{2}{3}\beta & -\tfrac{7}{3}& 8 &  & \\[-0.1cm]
	0 & 0 & 0 & 0 & 18\alpha^2+7\alpha  & 3\alpha  & 9 & \\[-0.1cm]
	0 & 0 & 0 & 0 & -18\alpha^2\beta & -6\alpha \beta & -2\beta & 11\\       
	\end{pmatrix}\]}
\subsection{$\g=\mu_9^{20,\alpha,\beta}$ }\ 
{\tiny\[D=\begin{pmatrix}
	3 &  &  &  &  &  &  & \\[-0.1cm]
	0 & 4 &  &  &  &  &  & \\[-0.1cm]
	0 & 0 & 5 &  &  &  &  & \\[-0.1cm]
	0 & 0 & 0 & 6 &  &  &  & \\[-0.1cm]
	0 & 0 & 0 & -\tfrac{\alpha-\beta}{\alpha-2} & 7 &  &  & \\[-0.1cm]
	0 & 0 & 0 & 0 & -\tfrac{\beta}{\alpha} & 8 &  & \\[-0.1cm]
	0 & 0 & 0 & 0 &  \tfrac{\beta^2}{\alpha^2} & -\tfrac{\beta}{\alpha} & 9 & \\[-0.1cm]
	0 & 0 & 0 & 0 & - \tfrac{\beta^2}{\alpha^2} & \tfrac{\beta}{\alpha} & -1 & 10\\       
	\end{pmatrix}\]}

\begin{remark}
The algebras $\mu_9^{1,-1,\beta}, \mu_9^{1,0,\beta}, \mu_9^{1,1,\beta}, \mu_9^{1,\tfrac{1}{2},\beta}$ and  $ \mu_{17}^{(\alpha, \beta)}$ with  $\beta \neq 0$ doesn't fix in any our last  analysis,  in the first case because all the derivations of  $\mathfrak{h}_1$ are nilpotent, and in the second case because, despite $\mathfrak{h}_2$ have semi simple derivations, they have repeated eigenvalues. Thus formula \eqref{coeficiente r,s} is not valid. \\

The case $\mu_9^{1,-1,\beta}$ is specially remarkable since neither $\mathfrak{h}_1$ or $\mathfrak{h}_2$ have semisimple derivations.
\end{remark}

\section*{Acknowledgments}

Sonia V. is supported by MINEDUC-UA project, code ANT 1999.

\begin{table}{Appendix:\label{coeficientesFiliformes}\\
We give all the explicit values of $ C^9_{r,s}$ for each non isomorphic  filiform Lie Algebra using the classification in \cite{GJMK}.}
\begin{center}\tiny{
\begin{tabular}{|c|ccccccccc|}\hline
	$\mu_9^{i}$  &  $C^9_{4,5}$ &  $C^9_{3,6}$ &  $C^9_{2,7}$ & $C^9_{3,5}$ & $C^9_{2,6}$ & $C^9_{3,4}$ &  $C^9_{2,5}$ & $C^9_{2,4}$ &  $C^9_{2,3}$\\\hline
	1 & $\frac{3 \alpha^2}{\alpha +2}$ & $\alpha-\frac{3 \alpha^2}{\alpha +2}$ & $\frac{2-5 \alpha}{\alpha +2}$ & 1 & -2 & $\beta $ & $-\beta$  & 0 & 0 \\\hline
	2 & $\frac{3 \alpha^2}{\alpha +2}$ & $\alpha-\frac{3 \alpha^2}{\alpha +2}$ & $\frac{2-5 \alpha}{\alpha +2}$ & 0 & 0 & 1 & -1 & 0 & 0 \\\hline
	4 & 8 & -4 & -3 & $\alpha$ & -2 $\alpha$ & $\beta$  & $-\beta$  & 0 & 1 \\\hline
	5 & 8 & -4 & -3 & 0 & 0 & $\alpha$ & $-\alpha$ & 0 & $\alpha$ \\\hline
	6 & 8 & -4 & -3 & 0 & 0 & 0 & 0 & 0 & 1 \\\hline
	7 & 3 & -2 & 0 & 0 & 1 & $\alpha$ & $-\alpha$ & 0 & 0 \\\hline
	8 & 3 & -2 & 0 & 0 & 0 & 1 & -1 & 0 & 0 \\\hline
	10 & 1 & -1 & 1 & 1 & $\alpha-2$ & 0 & $\beta$  & 0 & 0 \\\hline
	11 & 1 & -1 & 1 & 1 & -1 & $\beta$  & $\alpha-\beta$  & 0 & 0 \\\hline
	12 & 1 & -1 & 1 & 1 & $-\frac{7}{3}$ & $\beta$  & $\alpha-\beta$  & 0 & 0 \\\hline
	13 & 1 & -1 & 1 & 0 & 1 & 0 & $\alpha$ & 0 & 0 \\\hline
	14 & 1 & -1 & 1 & 0 & 0 & 0 & 1 & $\alpha$ & 0 \\\hline
	15 & 1 & -1 & 1 & 0 & 0 & 0 & 0 & 1 & 0 \\\hline
	17 & 0 & 0 & 1 & 1 & -2 & $\beta$  & $\alpha-\beta$  & 0 & 0 \\\hline
	18 & 0 & 0 & $\alpha$ & 0 & 0 & 1 & $\beta-1$ & 1 & 0 \\\hline
	19 & 0 & 0 & 1 & 0 & 0 & 1 & $\alpha-1$ & 0 & 0 \\\hline
	20 & 0 & 0 & 0 & 1 & $\alpha-2$ & 1 & $\beta-1$ & 0 & 0 \\\hline
	21 & 0 & 0 & 0 & 1 & $\alpha-2$ & 0 & 1 & 0 & 0 \\\hline
	23 & 0 & 0 & 0 & 0 & 1 & 1 & $\alpha-1$ & 0 & 0 \\\hline
	24 & 0 & 0 & 0 & 0 & 0 & 1 & $\alpha-1$ & 1 & 0 \\\hline
	26 & 0 & 0 & $\alpha$ & 0 & 0 & 0 & 1 & 1 & $\beta$  \\\hline
	27 & 0 & 0 & 1 & 0 & 0 & 0 & 1 & 0 & $\alpha$ \\\hline
	28 & 0 & 0 & $\alpha$ & 0 & 0 & 0 & 0 & 1 & 1 \\\hline
	29 & 0 & 0 & 1 & 0 & 0 & 0 & 0 & 1 & 0 \\\hline
	30 & 0 & 0 & 1 & 0 & 0 & 0 & 0 & 0 & 1 \\\hline
	31 & 0 & 0 & 0 & 0 & 1 & 0 & 1 & 0 & $\alpha$ \\\hline
	32 & 0 & 0 & 0 & 0 & 1 & 0 & 0 & 0 & 1 \\\hline
	34 & 0 & 0 & 0 & 0 & 0 & 0 & 1 & 0 & 1 \\\hline
\end{tabular}
}
\end{center}
\end{table}



\end{document}